\documentclass[11pt, twoside]{article}

\usepackage{fancyhead}
\usepackage{graphicx}
\usepackage{amssymb, amsmath, amsthm}
\usepackage{url}

\newcommand{\ceil }[1]{\left\lceil #1\right\rceil}

\theoremstyle{definition}
\newtheorem{definition}{Definition}
\theoremstyle{plain}
\newtheorem{lemma}{Lemma}
\theoremstyle{plain}
\newtheorem{theorem}{Theorem}

\begin{document}

\vspace*{5mm}

\normalsize

\noindent
{\textbf{{\LARGE Renyi-Ulam Games and Forbidden Substrings}}}

\thispagestyle{fancyplain} \setlength\partopsep{0pt} \flushbottom
\date{}

\vspace*{5mm}
\noindent

\textsc{Nikolai Beluhov}

\medskip

\begin{center}
\parbox{11,8cm}{\footnotesize
\textbf{Abstract.} The Renyi-Ulam game is played between two players, the Seeker and the Obscurer. The Obscurer thinks of a number between 1 and $n$. The Seeker wishes to identify that number. On each turn, the Seeker asks the Obscurer whether her number belongs to a specific subset of the numbers from 1 to $n$. The Obscurer answers either yes or no, and her answer is either true or false. The series of truths and lies given by the Obscurer must conform to a restriction $R$ that the players have agreed on in advance.

We give criteria on the restrictions $R$ that allow the Seeker to win. Then we apply our results to the study of restrictions characterized by forbidden substrings. In particular, we give a complete classification of all such restrictions characterized by two forbidden substrings, elaborating on Czyzowicz, Lakshmanan and Pelc's classification of all such restrictions characterized by one forbidden substring.}

\end{center}

\baselineskip=0.9\normalbaselineskip

\section{Introduction}\label{intro}

Renyi-Ulam games were first introduced by Alfred Renyi in 1961 in \cite{Renyi}. Stanislaw Ulam rediscovered the game and initiated its popularity in 1976 in \cite{Ulam}.

In its original form, the Renyi-Ulam game is played between two players, a Seeker and an Obscurer, as follows. The Seeker seeks to identify an object, say, a number between 1 and $n$, known to the Obscurer. To this end, he may ask arbitrary yes-no questions. The Obscurer replies, for the most part, accurately -- with the exception of at most one lie. What is the minimum number of questions that the Seeker needs to attain his goal, and what strategy should he follow?

Many variations on this basic theme have been explored in the literature. A thorough survey is \cite{Pelc}.

We focus on Renyi-Ulam games in which the Obscurer's capacity to mislead the Seeker is restricted in a different way but the rules of the game are otherwise unchanged. For instance, the Obscurer could agree to never follow a lie with a truth, as in \cite{CLP}.

In section \ref{ingen} we study whether, given a specific restriction on the Obscurer's inaccuracy, the Seeker can win at all. Then in section \ref{substr} we apply the tools we have built to the family of restrictions characterized by forbidden substrings in the series of truths and lies.

\section{Renyi-Ulam games in general}\label{ingen}

\begin{definition}

A \emph{lie restriction} $R$ is a condition on binary strings.

\end{definition}

A 0 in a binary string stands for a truth and a 1 stands for a lie. We shall only study lie restrictions $R$ such that:

\begin{itemize}
\item $R$ is \emph{prefix-closed}, that is, every prefix of a string that satisfies $R$ also satisfies $R$, and
\item $R$ is \emph{extensible}, that is, every string that satisfies $R$ can be extended by one letter so as to continue to satisfy $R$.
\end{itemize}

This is a natural requirement. As a string that satisfies a restriction $R$ is formed over the course of a game, $R$ needs to be prefix-closed (for, otherwise, the Obscurer has already broken $R$ on a previous move) and extensible (for, otherwise, the Obscurer would have no legal reply should the Seeker continue to ask questions).

Here are some restrictions that have been looked into. See \cite{Pelc} for details.

\begin{itemize}
\item The string $s$ contains at most $c$ lies, where $c$ is a fixed positive integer.
\item Every prefix $p$ of the string $s$ contains at most $c \cdot |p|$ lies, where $c$ is a fixed positive real.
\item Every substring $p$ of the string $s$ such that $|p| = a$ contains at most $b$ lies, where $a$ and $b$ are fixed positive integers.
\item The string $s$ does not contain a substring equal to $p$, where $p$ is a fixed string.
\end{itemize}

\begin{definition}

The \emph{Renyi-Ulam game with lie restriction $R$} is played as follows.

There are two players: the Seeker and the Obscurer. The Obscurer thinks of a number between 1 and $n$. The Seeker wishes to identify that number.

On each turn, the Seeker asks the Obscurer whether her number belongs to a specific subset of the numbers from 1 to $n$. The Obscurer answers either yes or no, and her answer is either true or false. On every turn, the series of truths and lies given by the Obscurer thus far must belong to $R$.

\end{definition}

Consider the Seeker's state of knowledge by the end of turn $i$. For each binary string $s$ of length $i$ that satisfies $R$, the Seeker knows that if the accuracy of the Obscurer's replies thus far is described by $s$, then the Obscurer's number belongs to some subset $A(s)$ of the numbers from 1 to $n$ (where $A(s)$ depends on the particular series of questions asked by the Seeker).

Whenever the Seeker asks whether the Obscurer's number belongs to a set $Q$, each $A(s)$ is split: if the Obscurer answers ``yes'', then $A(s)$ is split into $A(s0) = A(s) \cap Q$ and $A(s1) = A(s) \cap \overline{Q}$, and if she answers ``no'', then $A(s)$ is split into $A(s0) = A(s) \cap \overline{Q}$ and $A(s1) = A(s) \cap Q$. Furthermore, whenever $s0$ or $s1$ does not satisfy $R$, the corresponding subset of possibilities drops out of consideration.

The Seeker's goal is to attain a state of knowledge such that all $A(s)$ but one are empty, the exceptional $A(s)$ containing exactly one element: the Obscurer's number. 

A key observation at this point is that, since the $A(s)$ are disjoint following each turn, all that matters is the number of elements in each $A(s)$ rather than which elements precisely comprise it. Therefore, it makes sense to instead consider $a(s)$, the number of possibilities for the Obscurer's number that each $s$ permits.

On turn 0, before the game begins, $a(\varepsilon) = n$. On each subsequent turn, the Seeker effectively picks two numbers $b(s)$ and $c(s)$ for each $a(s)$ so that $a(s) = b(s) + c(s)$, asks a question whose $Q$ intersects each $A(s)$ by $b(s)$ elements, and arrives at $a(s0) = b(s)$ and $a(s1) = c(s)$ if the Obscurer answers ``yes'' and at $a(s0) = c(s)$ and $a(s1) = b(s)$ if she answers ``no''. Again, whenever $s0$ or $s1$ does not satisfy $R$, the corresponding number of possibilities drops out of consideration. The Seeker's goal is to make all $a(s)$ but one equal to zero, with the exceptional $a(s)$ equal to one (or, equivalently, to bring the sum of all $a(s)$ down to unity).

This reduction makes it natural to group together descriptors $s$ that have identical possible futures. That is, if $s_1$ and $s_2$ are such that, for all strings $s'$, the strings $s_1s'$ and $s_2s'$ either both satisfy $R$, or both fail to satisfy $R$, then there is no need to deal with $a(s_1)$ and $a(s_2)$ separately: they can be replaced with their sum, $a(s_1) + a(s_2)$, and it can henceforth be treated as any other $a(s)$.

This motivates the following series of definitions.

\begin{definition}

A \emph{lie restriction graph} $G$ for a lie restriction $R$ is a directed graph $G$ such that:

\begin{itemize}
\item One vertex in $G$ is called a \emph{start vertex}.
\item Every arc in $G$ is labeled either 0 or 1.
\item Every vertex in $G$ originates at least one and at most two outward arcs. Of those, at most one arc is labeled 0 and at most one arc is labeled 1.
\item A binary string $s$ satisfies $R$ if and only if there is a path $l$ in $G$ such that $l$ starts from the start vertex and the arc labels along $l$ trace out $s$.
\end{itemize}

\end{definition}

Two strings $s_1$ and $s_2$ have identical possible futures iff the paths that they define in $G$ starting from the start vertex end up on the same vertex. A restriction graph captures the notion of identical possible futures in precisely the same way as this is done in automata theory. More precisely, if $A$ is a deterministic automaton that recognizes the language of all strings that satisfy $R$ then removing all non-accept states from $A$ results in a restriction graph for $R$.

Every prefix-closed extensible restriction $R$ possesses at least one restriction graph $G$: the vertices of $G$ are all strings $s$ that satisfy $R$, and an arc labeled $a$ in $G$ always points from a string $s$ to $sa$. Of course, given a prefix-closed extensible restriction $R$, there may be many distinct restriction graphs $G$ for $R$.

Consider, then, a fixed restriction $R$ and a fixed restriction graph $G$ for $R$.

\begin{definition}

A \emph{position} is a mapping from the vertices of $G$ to the nonnegative integers.

\end{definition}

A position describes the Seeker's state of knowledge following a turn. We assume a numbering of the vertices of $G$ (with 1 being the start vertex) and write $(a_1, a_2, \ldots)$ for the position that maps vertex $i$ to the integer $a_i$.

\begin{definition}

The \emph{support} of a position $P$ is the number of vertices $v$ of $G$ such that $P(v)$ is nonzero.

\end{definition}

All positions that we are about to consider are of finite support.

\begin{definition}

The \emph{weight} of a position $P$ is the sum of $P(v)$ over all vertices $v$ of $G$.

\end{definition}

The Seeker's goal is to reach a position of unit weight. 

\begin{definition}

Let $P$, $P_1$, and $P_2$ be three positions. We say that $(P_1, P_2)$ is a \emph{split} of $P$ if there exist two positions $Q_0$ and $Q_1$ such that

\begin{itemize}
\item For all vertices $v$ of $G$, $P(v) = Q_0(v) + Q_1(v)$.
\item For all vertices $v$ of $G$, $P_1(v)$ is the sum of $Q_0(u)$ over all vertices $u$ of $G$ such that a 0-arc in $G$ points from $u$ to $v$.
\item For all vertices $v$ of $G$, $P_2(v)$ is the sum of $Q_1(u)$ over all vertices $u$ of $G$ such that a 1-arc in $G$ points from $u$ to $v$.
\end{itemize}

\end{definition}

A split captures the Seeker's capacity to advance to a different position by asking a question. If the Seeker is currently in a position $P$ and $(P_1, P_2)$ is a split of $P$, then the Seeker can devise a question that, depending on the Obscurer's answer, will advance the game to either $P_1$ or $P_2$.   

\begin{definition}

A \emph{strategy-like tree} is a binary tree such that its vertices are positions and every non-leaf vertex $P$ has two children that form a split of $P$. A strategy-like tree for a position $P$ is a strategy-like tree rooted at $P$.

\end{definition}

\begin{definition}

A \emph{strategy tree} is a strategy-like tree whose every leaf has weight either zero or one. A strategy tree for a position $P$ is a strategy tree rooted at $P$.

\end{definition}

A leaf of zero weight corresponds to a position that cannot occur in a game, and a leaf of unit weight corresponds to a position in which the Seeker has won.

The Seeker wins the Renyi-Ulam game with lie restriction $R$ from a position $P$ iff there is a strategy tree for $P$. Consequently, the Seeker wins the Renyi-Ulam game with lie restriction $R$ for $n$ iff there is a strategy tree for the position $(n, 0, 0, \ldots)$.

We are done setting up the scaffolding and proceed to give two criteria on winning positions.

\begin{theorem}\label{all}

The Seeker wins from all positions iff the Seeker wins from all positions of weight two.

\end{theorem}

\begin{proof}

Suppose that the Seeker wins from all positions of weight two. Then there is a strategy tree for every position of weight two. We will prove that there is a strategy tree for every position by induction on weight.

Let $P$ be a position such that $w(P) \ge 3$. Then either $P$ has a coordinate greater than or equal to two, or $P$ has at least two coordinates greater than or equal to one. In both cases, $P = P_1 + P_2$ for some positions $P_1$ and $P_2$ such that $w(P_1)$ and $w(P_2)$ are both strictly less than $w(P)$.

Let $T_1$ and $T_2$ be strategy trees for $P_1$ and $P_2$. Extend $T_1$ and $T_2$ to two isomorphic strategy trees $T'_1$ and $T'_2$ by adding on sufficiently many leaves of weight zero or one. Furthermore, let $T'$ be isomorphic to both $T'_1$ and $T'_2$ and every vertex of $T'$ be the sum of the corresponding vertices of $T'_1$ and $T'_2$. Then $T'$ is a strategy-like tree for $P$ all of whose leaves have weight at most two. Replace every leaf $L$ of $T'$ with a strategy tree for $L$, and the result is a strategy tree for $P$.

\end{proof}

\begin{theorem}\label{alln}

The Seeker wins from all positions of the form $(n, 0, 0, \ldots)$ iff the Seeker wins from all positions of the form $(0, \ldots, 0, 2, 0, \ldots)$ (unit support, weight two).

\end{theorem}

\begin{proof}

Suppose first that the Seeker wins from all positions of the form ($n$, 0, 0, \ldots). Let $u$ be a vertex in $G$. Let $a$ be the starting vertex: then there is a path $a$--$v_1$--$\ldots$--$v_k$--$u$ from $a$ to $u$.

Let $T$ be a strategy tree for the position $(2^{k + 2}, 0, 0, \ldots)$. One of the children $P_1$ of this position in $T$ has a $v_1$-coordinate of at least $2^{k + 1}$. One of the children of $P_1$, $P_2$, has a $v_2$-coordinate of at least $2^k$, and so on. Eventually, some vertex of $T$ has a $u$-coordinate of at least 2. This means that the Seeker wins from the position that maps $u$ to two and all other vertices to zero.

Suppose, then, that the Seeker wins from all positions of the form (0, \ldots, 0, 2, 0, \ldots). We set out to prove that the Seeker wins from all positions of the form (0, \ldots, 0, $n$, 0, \ldots) by induction on $n$.

Let $n \ge 3$ and $m = \ceil{\frac{n}{2}}$. Let $P$ be a position of the form $(0, \ldots, 0, n, 0, \ldots)$. Obtain $P'$ from $P$ by replacing $n$ with $m$ and obtain $P''$ from $P'$ by replacing $m$ with two.

Let $T''$ be a strategy tree for $P''$. Obtain $T'$ from $T''$ by multiplying all vertices of $T''$ by $m$. Then $T'$ is a strategy-like tree for $P'$ all of whose leaves are either of the form $(0, \ldots, 0, m, 0, \ldots)$ or $(0, 0, \ldots)$.

By induction hypothesis, it is possible to replace every leaf $L$ of $T'$ with a strategy tree for $L$, thus obtaining a strategy tree for $P'$. It follows that the Seeker wins from $P'$ and, therefore, from $P$.

\end{proof}

The situations discussed in Theorem \ref{all} and Theorem \ref{alln} are distinct: there are restrictions $R$ such that the Seeker wins from all unit-support weight-two positions but not from all weight-two positions. One example of this is given by the restriction graph $\{(1, 2), (1, 3), (2, 1), (3, 1)\}$ corresponding to the restriction that answers $2i - 1$ and $2i$ are either both truths or both lies for all $i$.

\begin{definition}

A lie restriction $R$ is \emph{regular} if the set of all strings that satisfy $R$ is a regular language. 

\end{definition}

A restriction $R$ is regular iff there is a finite restriction graph $G$ for $R$.

\begin{theorem}

Let $R$ be a regular restriction and $G$ be a finite restriction graph for $R$ with $k$ vertices. Suppose that the Seeker wins for $n = 2^k$. Then the Seeker wins for all $n$.

\end{theorem}

\begin{proof}

As in the first part of the proof of Theorem \ref{alln}, given that the Seeker wins for all $n = 2^k$, we conclude that the Seeker wins from all positions of the form $(0, \ldots, 0, 2, 0, \ldots)$. By Theorem \ref{alln}, the proof is complete.

\end{proof}

\begin{theorem}

Let $R$ be a regular restriction such that the Seeker wins for all $n$. Then the Seeker wins in $O(\log n)$ steps.

\end{theorem}

\begin{proof}

Let $c$ be such that the Seeker wins all positions of the form (0, \ldots, 0, 2, 0, \ldots) in at most $c$ steps and $h(n)$ be the number of iterations of the transform $n \to \ceil{\frac{n}{2}}$ that it takes to boil $n$ down to $1$. The construction in the second part of the proof of Theorem \ref{alln} shows that there is a strategy tree for the position ($n$, 0, 0, \ldots) of depth at most $c \cdot h(n)$, and this is $O(\log n)$.

\end{proof}

There are non-regular restrictions $R$ such that the Seeker wins for all $n$ but cannot do so in a logarithmic number of steps. For instance, construct a restriction graph $G$ for $R$ as follows:

\begin{itemize}
\item The vertices of $G$ are $(a, b)$, where $a$ ranges over all positive integers and $1 \le b \le F(a)$ where $F$ grows very rapidly.
\item $(a, 1)$ points to $(a + 1, 1)$ and $(a, 2)$ for all $a$.
\item For all $b \ge 2$, $(a, b)$ points to itself and, as long as $b < F(a)$, to $(a, b + 1)$.
\end{itemize}

We proceed to describe an equivalent form of the criterion given in Theorem \ref{alln} which shall prove useful in section \ref{substr}.

\begin{definition}

The \emph{two-string game} for a lie restriction $R$ is played between two players as follows. The starting position is given by two binary strings $u_1$ and $u_2$. On each turn, the Obscurer extends one of $u_1$ and $u_2$ (that is, appends a character to it) and the Seeker replies by extending the other. The Seeker's goal is that one of the two strings ceases to satisfy $R$ (it does not matter whether this happens following a move by the Seeker or the Obscurer) and the Obscurer's goal is to prevent this.

\end{definition}

We shall freely identify strings that satisfy a restriction $R$ with vertices in a restriction graph $G$ for $R$. Moves in the two-string game can thus also be regarded as transitions within $G$. 

\begin{theorem}\label{twostrl}

The Seeker wins the Renyi-Ulam game with lie restriction $R$ from a position $P$ of weight two iff the Seeker wins the two-string game for $R$ from a starting position given by the $P$-nonzero vertices of the lie restriction graph $G$ (or, should there be just one such vertex $v$, by two copies of $v$). 

\end{theorem}

\begin{proof}

Let $P_{x, y}$ be the position of weight two with nonzero coordinates $x$ and $y$ (if $x \neq y$ then both the $x$ and $y$ coordinate of $P_{x, y}$ equal one, and if $x = y$ then the $x$-coordinate of $P_{x, y}$ equals two).

Obtain $H$ from $G$ by adjoining an additional \emph{error vertex} $e$ such that:

\begin{itemize}
\item for every vertex $x$ of $G$ and every label $a$ such that in $G$ no arc labeled $a$ originates at $v$, in $H$ an arc labeled $a$ points from $v$ to $e$, and
\item in $H$ two arcs labeled 0 and 1 point from $e$ to $e$.
\end{itemize}

Then in $H$ every vertex has out-degree two.

Consider $P = P_{u, v}$. Let, in $H$, $a$ and $b$ be the children of $u$ and $c$ and $d$ be the children of $v$.

Suppose that the Obscurer wins the Renyi-Ulam game with restriction $R$ from $P_{u, v}$. There are only two possibilities for the children of $P_{u, v}$ in a strategy-like tree: $P_{a, c}$ and $P_{b, d}$, or $P_{a, d}$ and $P_{b, c}$. Since $P_{u, v}$ is a (Renyi-Ulam) win for the Obscurer, (at least) one position in the first pair and (at least) one position in the second pair are also (Renyi-Ulam) wins for the Obscurer. Without loss of generality, those are $P_{a, c}$ and $P_{a, d}$: in this case, we say that $a$ is a key vertex for $P_{u, v}$.

Consider the two-string game for $R$ played from $P_{u, v}$. By always moving to the key vertex for the current position, the Obscurer attains that the position following the Seeker's move is always a (Renyi-Ulam) win for the Obscurer. This ensures that none of the two strings ever reaches the error vertex.

For the converse, let $P_{u, v}$ be an arbitrary position such that the Obscurer wins the two-string game for $R$ starting from $P_{u, v}$. Suppose that transitioning from $u$ to $a$ is a (two-string) winning move for the Obscurer. Then, in the Ulam-Renyi game with lie pattern $R$, it is possible that the Obscurer replies to the Seeker's question in such a way that the resulting position has a nonzero $a$ coordinate. Since, in the two-string game, no string ever reaches the error vertex, this ensures that no position of weight one ever occurs.

\end{proof}

\begin{theorem}\label{twostr}

The Seeker wins the Renyi-Ulam game with lie restriction $R$ from all positions of the form ($n$, 0, 0, \ldots) iff the Seeker wins the two-string game for $R$ from all starting positions such that the two initial strings $u_1$ and $u_2$ are equal.

\end{theorem}

\begin{proof}

This is a strightforward corollary of Theorem \ref{alln} and Theorem \ref{twostrl}.

\end{proof}

Notice that, when $G$ is finite, the number of essentially different positions in the two-string game for $R$ is bounded by the square of the number $k$ of vertices in $G$. Consequently, when $k$ is small, it is feasible to determine who has a winning strategy in the two-string game for $R$ by means of direct computation. We shall make frequent use of this technique in the following section.

\section{Forbidden substrings}\label{substr}

Let $S$ be a set of binary strings and $R = r(S)$ be the restriction that a binary string $s$ does not contains a substring in $S$. Given $S$, who wins in a Renyi-Ulam game with restriction $R$?

The case $|S| = 1$ was first studied, and solved completely, by Czyzowicz, Lakshmanan and Pelc in \cite{CLP}. Theorem \ref{one} below is the main result in their paper.

\begin{definition}

We say that $S_1$ is \emph{expanded} by $S_2$ if $S_1$ contains a substring of every string in $S_2$.

\end{definition}

In this case, $R_1 = r(S_1) \subseteq R_2 = r(S_2)$ and $R_1$ is more restrictive on the Obscurer's capacity to mislead the Seeker than $R_2$. This means that, if the Obscurer wins for $R_1$, then she also wins for $R_2$. Conversely, if the Seeker wins for $R_2$, then he also wins for $R_1$.

We set out to find a natural sufficient condition for the Obscurer to win. By Theorem \ref{twostr}, it suffices to find a sufficient condition for the Obscurer to win the two-string game for $R$ from some starting position comprised of two identical strings.

\begin{definition}

The \emph{risk} of a string $p$ is the length of the longest suffix of $p$ that is a prefix of an element of $S$. When $p$ is of risk $l$ and the $l$-suffix of $p$ is an $l$-prefix of $q \in S$, we say that $p$ is of risk $l$ \emph{via} $q$. The \emph{risk} of a position $P$ in the two-string game for $R$ is the maximum of the risks of the two strings that comprise $P$.

\end{definition}

Notice that extending a string increases its risk by at most one.

Suppose that all strings in $S$ are of length at least $l + 1$ and that (starting from a position of risk at most $l - 1$, say, two empty strings) the Obscurer can play so that, following her every move, the risk of the position is at most $l - 1$. Then the Obscurer wins.

What is a natural sufficient condition for that?

\begin{definition}

The \emph{$l$-flip} of a string $p = a_1a_2 \ldots a_m$ of length $m \ge l + 1$ is the string $a_1a_2 \ldots a_l\overline{a_{l + 1}}$.

\end{definition}

Suppose that, for some $l \ge 1$, all $(l - 1)$-flips and all $l$-flips of elements of $S$ are of risk at most $l - 1$. Then the Obscurer can keep risk following her moves below or at the $l - 1$ level by always spoiling the riskier string. This is the condition we were looking for, and its negation is a necessary condition for the Seeker to win.

\begin{lemma}\label{flip}

For the Obscurer to win the Renyi-Ulam game with lie restriction $R = r(S)$, it is sufficient that there is a positive integer $l$ such that all $(l - 1)$-flips and all $l$-flips of elements of $S$ are of risk at most $l - 1$.

For the Seeker to win the Renyi-Ulam game with lie restriction $R = r(S)$, it is necessary that, for all $l \ge 1$, there is either some $(l - 1)$-flip or some $l$-flip of an element of $S$ of risk at least $l$.

\end{lemma}

Let us test the strength of this criterion on a known result.

\begin{theorem}\label{one}

When $S = \{s_1\}$, the Seeker wins iff $s_1$ is one of $0$, $1$, $01$, and $10$.

\end{theorem}

\begin{proof}

Notice first that the $l$-flip of $s_1$ is not a prefix of $s_1$ so that its risk does not exceed $l$.

Suppose that $|s_1| \ge 3$, $s_1 = a_1a_2 \ldots a_m$, and, in Lemma \ref{flip}, set $l = 2$. Then the 2-flip of $s_1$ must be of risk 2, which gives us $a_2\overline{a_3} = a_1a_2$.

Set, then, $l = 1$ in Lemma \ref{flip}: analogously, the 1-flip of $s_1$ needs to be of risk 1, giving us $\overline{a_2} = a_1$, a contradiction.

It follows that $|s_1| \le 2$ and direct computation yields the claim.

\end{proof}

We turn to a complete description of all $S = \{s_1, s_2\}$ such that the Seeker wins. In the series of lemmas that follow, we assume that $S = \{s_1, s_2\}$ and $|s_1| \le |s_2|$ except where noted otherwise.

\begin{lemma}\label{comm}

If the Seeker wins, then  the strings $s_1$ and $s_2$ do not share a common substring $p$ equal to $00$ or $11$ or of length at least three.

\end{lemma}

\begin{proof}

For, otherwise $\{p\}$ is expanded by $S$ and Theorem \ref{one} gives us a contradiction.

\end{proof}

\begin{lemma}\label{len}

If the Seeker wins, then the length of $s_1$ does not exceed four.

\end{lemma}

\begin{proof}

Suppose that $5 \le |s_1| \le |s_2|$ and apply Lemma \ref{flip} with $l = 4$.

Suppose that a 3-flip of $s_1$ is of risk 4. It cannot be so via $s_1$ because of the definition of a flip, nor can it be so via $s_2$ because of Lemma \ref{comm}: a contradiction. Analogously, a 3-flip of $s_2$ cannot be of risk 4.

Suppose, then, that a 4-flip of $s_1$ is of risk at least 4. It cannot be so via $s_2$ because of Lemma \ref{comm}. Therefore, it must be so via $s_1$, which gives us $s_1 = 00001\ldots$ modulo inverting all characters. (The case when a 4-flip of $s_2$ is of risk 4 is treated analogously.)

Next we apply Lemma \ref{flip} with $l = 3$.

As above, no 2-flip of either string may be of risk 3 and no 3-flip of either string may be of risk 3 via the other (it is important at this point that we know that $s_1$ starts on a series of equal characters so that Lemma \ref{comm} applies).

If a 3-flip of $s_1$ is of risk 3 via $s_1$, we arrive at a contradiction. Therefore, a 3-flip of $s_2$ is of risk 3 via $s_2$ and $s_2 = 0001\ldots$ or $s_2 = 1110\ldots$.

When $s_2 = 0001\ldots$, Lemma \ref{comm} gives us a contradiction. When $s_2 = 1110\ldots$, Lemma \ref{flip} gives us a contradiction with $l = 2$.

\end{proof}

Notice that we could have obtained the same result by testing all pairs of length-5 strings directly. The above argument, however, generalizes: see Theorem \ref{lenn}.

\begin{lemma}\label{len4}

When $|s_1| = 4$, the Seeker wins for the following sets $S$: $\{0001$, $1011\}$, $\{0001$, $1101\}$, $\{0010$, $0111\}$, $\{0011$, $0101\}$, $\{0100$, $0111\}$, and the sets obtained from them by inverting all characters.

\end{lemma}

\begin{proof}

Direct computation shows that there are no other $S$ such that the Seeker wins when $|s_1| = 4$ and $|s_2| \le 5$. Since every $S$ such that $4 = |s_1| \le |s_2|$ expands some $S$ such that $4 = |s_1| \le |s_2| \le 5$, the claim follows.

\end{proof}

\begin{lemma}\label{000}

When $s_1 = 000$, the Seeker wins for the following $s_2$: $011$, $101$, $110$, $1011$, and $1101$.

\end{lemma}

\begin{proof}

Analogous to the proof of Lemma \ref{len4}.

\end{proof}

Next comes a brief detour whose purpose will become evident later.

\begin{lemma}\label{gcdd}

When $d \ge 2$ and $S = \{01^{kd}0 \,|\, k \ge 0\}$, the Obscurer wins from the empty-string starting position. 

\end{lemma}

\begin{proof}

If $u_1$ does not contain any 0s, set $h_1 = \infty$. Otherwise, set $h_1$ equal to the remainder of the number of 1s since the last 0 in $u_1$ upon division by $d$. Define $h_2$ analogously for $u_2$.

We say that a position is \emph{safe} if it is the Seeker to move, both $h_1$ and $h_2$ differ from 0, and, additionally, if $h_i + 1 = h_{\sim i} \neq \infty$, then the Seeker is to extend $h_{\sim i}$. (Here, $\sim 1 = 2$ and $\sim 2 = 1$.)

The Obscurer's strategy is to always write down a 1 so as to make the position safe. We proceed to show that this is indeed possible by means of induction and casework.

The Obscurer's opening move results in a safe position with $h_1 = h_2 = \infty$.

Suppose that the Seeker is to move in a safe position with $h_1 \le h_2$. We write Player: $i \circ a$ for Player extending $u_i$ by $a$. Here follows a list of all possible situations together with viable replies for the Obscurer. Immediately forced timelines of the form

\smallskip

\indent Seeker: $i$ $\circ$ any leading to $h_i = 0$, Obscurer: $i \circ 1$

\smallskip

\noindent are omitted for brevity.

\begin{itemize}
\item $h_1 + 1 < h_2 < d - 1$, Seeker: $i \cdot 1$, Obscurer: $\sim i \cdot 1$
\item $h_1 + 1 = h_2 < d - 1$, Seeker: $2 \cdot 1$, Obscurer: $1 \cdot 1$
\item $h_1 = h_2 < d - 1$, Seeker: $i \cdot 1$, Obscurer: $\sim i \cdot 1$
\item $h_1 + 1 < h_2 = d - 1$, Seeker: $1 \cdot 1$, Obscurer: $1 \cdot 1$
\item $h_1 + 1 = h_2 = d - 1$
\item $h_1 = h_2 = d - 1$
\item $h_1 < h_2 = \infty$, Seeker: $i \cdot 1$ not leading to $h_1 = 0$, Obscurer: $2 \cdot 1$
\item $h_1 = h_2 = \infty$, Seeker: $i \cdot 1$, Obscurer: $i \cdot 1$.
\end{itemize}

All of those lead back to a safe position, and we are done.

\end{proof}

A straightforward corollary is

\begin{lemma}\label{01k0}

When $k \ge 2$ and $S = \{00, 01^k0\}$, the Obscurer wins from the empty-string starting position.

\end{lemma}

Let $R_\#$ be the string set $r(\{01^k0 \,|\, k = 0 \textrm{ or } k \ge 2\})$. Then \[ R_\# = \{1^a\underbrace{010 \ldots 10}_b1^c \,|\, a, b, c \ge 0, b = 0 \textrm{ or } b \textrm{ is odd}\}. \]

\begin{lemma}\label{rsharp}

When $S = \{00, s_2\}$ with $s_2$ in $R_\#$, the Seeker wins.

\end{lemma}

\begin{proof}

Notice first that the Obscurer must only write down 1s. For, to Obscurer: $i \circ 0$, the Seeker replies $\sim i \circ 0$ and wins on the following turn.

Also, Seeker: $i \circ 0$ forces Obscurer: $i \circ 1$ -- for, if the Obscurer replies in any other way, the Seeker wins on the turn.

Take, then, $s_2 = 1^a\underbrace{010 \ldots 10}_b1^c$. The Seeker's strategy is as follows.

\begin{itemize}
\item If $b = 0$, the Seeker writes down 1s only until $s_2$ occurs.
\item From this point on, $b \ge 1$. The Seeker writes down 1s only until he is to extend some $u_i$, say $u_1$, that ends in at least $a$ 1s.
\item The Seeker loops over [Seeker: $1 \circ 0$, Obscurer: $1 \circ 1$, Seeker: $2 \circ 0$, Obscurer: $2 \circ 1$] $\frac{b + 1}{2}$ times.
\item The Seeker extends $u_1$ by 1s only and $u_2$ by 0s only until $s_2$ occurs in $u_1$.
\end{itemize}

\end{proof}

Lemma \ref{01k0} and Lemma \ref{rsharp} give us

\begin{lemma}\label{00}

When $S = \{00, s_2\}$, the Seeker wins iff $s_2 \in R_\#$.

\end{lemma}

\begin{lemma}\label{001}

When $s_1 = 001$, the Seeker wins iff $s_2 \in R_\#$.

\end{lemma}

\begin{proof}

By Lemma \ref{00}, this condition is necessary. We proceed to show that it is also sufficient.

By Lemma \ref{00}, the Seeker can play so that either 00 or $s_2$ occurs following a move by the Seeker. If the latter, we are done. If the former, then without loss of generality 00 has occurred in $u_1$. The following scenario tree,

\begin{itemize}
\item Obscurer: $1 \circ 0$, Seeker: $2 \circ 0$,
      \begin{itemize}
      \item Obscurer: $1 \circ 0$, Seeker: $2 \circ 0$, Obscurer: any, Seeker: any $\circ$ 1
      \item Obscurer: 2 $\circ$ any, Seeker: $1 \circ 1$
      \end{itemize}
\item Obscurer: 2 $\circ$ any, Seeker: $1 \circ 1$
\end{itemize}

\noindent (timelines such that the Seeker wins on the Obscurer's move are omitted for brevity) shows that the Seeker wins in at most two more turns.

\end{proof}

\begin{lemma}\label{010}

When $s_1 = 010$, the Seeker wins for the following $s_2$: $001$, $011$, $100$, $110$, $111$, $0011$, $0111$, $1100$, $1110$, $00111$, and $11100$.

\end{lemma}

\begin{proof}

Analogous to the proofs of Lemma \ref{len4} and Lemma \ref{000}.

\end{proof}

\begin{lemma}\label{100}

When $s_1 = 100$, the Seeker wins iff $s_2 \in R_\#$.

\end{lemma}

\begin{proof}

By Lemma \ref{00}, this condition is necessary. We proceed to show that it is also sufficient.

Initially, the Seeker ensures that both $u_1$ and $u_2$ contain a 1. Here follows one way to achieve this: without loss of generality, the Seeker is to extend $u_1$. He plays $1 \circ 1$, then proceeds to play $1 \circ 0$ and $1 \circ 0$ until he either wins by completing 100 in $u_1$ or the Obscurer interrupts him, giving him the opportunity to play $2 \circ 1$.

Following this, by Lemma \ref{00}, the Seeker plays so that either 00 or $s_2$ occurs. If the latter, we are done. If the former, then we are done as well because a 1 preceding 00 yields an occurrence of 100.

\end{proof}

All remaining $|s_1| \le 3$ cases are reduced to the ones described above by inverting all characters in $S$.

All cases such that $|s_1| \le 2$ are handled by Theorem \ref{one}, Lemma \ref{00}, and, for $s_2 = 11$, inverting all characters.

This completes the classification, and the following theorem summarizes it.

\begin{theorem}\label{two}

The complete list of all two-element sets of binary strings $S$ such that the Seeker wins the Ulam-Renyi game with lie restriction $R = r(S)$ is given by all sets expanded by a single-element Seeker-win set as listed in Theorem \ref{one}, the 28 sporadic pairs

\smallskip

\indent $s_1 = 000$, $s_2 = 101$, $1011$, $1101$,

\indent $s_1 = 010$, $s_2 = 111$, $0011$, $0111$, $1100$, $1110$, $00111$, $11100$,

\indent $s_1 = 101$, $s_2 = 0001$, $0011$, $1000$, $1100$, $00011$, $11000$,

\indent $s_1 = 111$, $s_2 = 0010$, $0100$,

\indent $\{0001, 1011\}$, $\{0001, 1101\}$, $\{0010, 0111\}$, $\{0010, 1110\}$, $\{0011, 0101\}$,

\indent $\{0100, 0111\}$, $\{0100, 1110\}$, $\{1000, 1011\}$, $\{1000, 1101\}$, $\{1010, 1100\}$,

\smallskip

\noindent and the six infinite families

\smallskip

\indent $s_1 \in \{00, 001, 100\}, s_2 \in R_\# \textrm{ and } s_1 \in \{11, 011, 110\}, s_2 \in R^{inv}_\#$.

\end{theorem}

A dual classification is given by the minimal (with respect to shrinking any element to a proper substring) Obscurer-win two-element $S$. It consists of 16 sporadic pairs,

\smallskip

\indent $s_1 = 000$, $s_2 = 010$, 111, 11011,

\indent $s_1 = 010$, $s_2 = 101$, 0110, 1001, 1111, 01110,

\indent $s_1 = 101$, $s_2 = 111$, 0000, 0110, 1001, 10001,

\indent $s_1 = 111$, $s_2 = 00100$,

\indent $\{0011, 1010\}$, and $\{0101, 1100\}$,

\smallskip

\noindent together with two infinite families,

\smallskip

\indent $\{00, 01^k0\}$ and $\{11, 10^k1\}$, $k \ge 2$.

\smallskip

By Theorem \ref{two}, the language formed by all strings $s_1 * s_2$ (where $*$ is a new character distinct from 0 and 1) such that the Seeker wins for $S = \{s_1, s_2\}$ is regular. As the following theorem shows, an analogous claim does not hold for three-element $S$.

\begin{theorem}\label{gcd}

When $S = \{00, 01^m0, 01^n0\}$, $m, n \ge 1$, the Seeker wins the Renyi-Ulam game with lie restriction $R = r(S)$ iff $m$ and $n$ are coprime.

\end{theorem}

\begin{proof}

Suppose first that $(m, n) = d > 1$. By Lemma \ref{gcdd}, the Obscurer wins from the empty-string starting position.

Suppose, then, that $(m, n) = 1$. We proceed to show that the Seeker wins by induction on $\max\{m, n\}$.

When one of $m$ and $n$ equals 1, the claim follows by Theorem \ref{two}. Assume, from this point on, that $2 \le m < n$.

As in the proof of Lemma \ref{rsharp}, the Obscurer must only write down 1s, and Seeker: $i \circ 0$ forces Obscurer: $i \circ 1$.

The Seeker begins by writing down 1s only for the first $n + 1$ turns. This ensures that, from turn $n + 2$ on, the game proceeds as if from the empty-string starting position.

By induction hypothesis, the Seeker wins for $S' = \{00, 01^m0, 01^{n - m}0\}$. Let the Seeker follow a strategy that wins for $S'$ from the empty-string starting position. Then there must come a moment when the Seeker is to move, the Seeker has not won yet (for $S'$) and cannot win on this turn (for $S'$), and the Seeker can play so as to win on the next turn (for $S'$) regardless of the Obscurer's reply.

At this point, one turn short of the final stroke, the Seeker abandons the $S'$ strategy and focuses back on his original goal, winning for $S$.

In a position $u_1$, $u_2$ such that it is the Obscurer to move, the Seeker has not won yet (for $S'$), and the Seeker wins on the turn (for $S'$), both of $u_1$ and $u_2$ must necessarily end in one of $0$, $01^m$, or $01^{n - m}$.

This means that, without loss of generality, at present $u_1$ ends in one of $0$, $01^m$, or $01^{n - m}$ and the Seeker is to extend $u_2$.

The Seeker plays $2 \circ 0$, forcing Obscurer: $2 \circ 1$. If $u_1$ ends in $0$ or $01^m$, the Seeker wins on his next move.

Otherwise, $u_1$ ends in $01^{n - m}$. The Seeker then writes down 1s only on his next move as well as on the $m - 1$ turns that follow it. This brings the Seeker and the Obscurer to a point when it is the Obscurer to move, $u_1$ ends in $01^n$, and $u_2$ ends in $01^m$. The Seeker wins on the turn.

\end{proof}

Theorem \ref{gcd} hints that there may be a lot more to uncover on forbidden string sets $S$ with $|S| \ge 3$. We conclude with a generalization of Lemma \ref{len} that applies to all $|S| \ge 2$.

\begin{theorem}\label{lenn}

When $|S| = k \ge 2$ and the Seeker wins the Renyi-Ulam game with lie restriction $R = r(S)$, the length of the shortest string in $S$ does not exceed $4k - 4$.

\end{theorem}

\begin{proof}

We proceed by induction on $k$. When $k = 2$, the claim holds by Lemma \ref{len}.

Let $f(k) = 4k - 4$, and suppose that all elements of $S$ are of length at least $f(k - 1) + 5$.

No two elements of $S$ share a common substring of length at least $f(k - 1) + 1$. For, otherwise, we could obtain $S'$ from $S$ by replacing them with this common substring and, since $S'$ is expanded by $S$, arrive at a contradiction by the induction hypothesis.

Apply, then, Lemma \ref{flip} with $l \ge f(k - 1) + 2$.

Suppose first that an $(l - 1)$-flip of $s_i$ is of risk $l$. Then it cannot be so via $s_i$ because of the definition of a flip, nor can it be so via $s_j$, $j \neq i$, because that would give us a long common substring: a contradiction.

Suppose, then, that an $l$-flip of $s_i$ is of risk at least $l$. It cannot be so via $s_j$, $j \neq i$, because that would give us a long common substring. Therefore, it must be so via $s_i$, which gives us $s_i = 0^l1\ldots$ or $s_i = 1^l0\ldots$.

Apply the previous argument with $l$ set successively to $f(k - 1) + 2$, $f(k - 1) + 3$, and $f(k - 1) + 4$ to obtain three distinct elements of $S$ starting on exactly $f(k - 1) + 2$, $f(k - 1) + 3$, and $f(k - 1) + 4$ identical characters. Two of those elements share a common prefix of length at least $f(k - 1) + 2$, and this gives us a long common substring: a contradiction.

Therefore, at least one string in $S$ is of length at most $f(k - 1) + 4 = f(k)$.

\end{proof}

\bibliographystyle{plain}
\bibliography{guessing-game} 

\end{document}